\documentclass[12pt]{article}
\RequirePackage[colorlinks,citecolor=blue,urlcolor=blue,linkcolor=blue]{hyperref}
\hypersetup{
colorlinks = true,
citecolor=blue,
urlcolor=blue,
linkcolor=blue,
pdfauthor = {Alexey Kuznetsov},
pdfkeywords = {Hamburger's indeterminate moment problem, Nevanlinna parametrization, entire functions,
 Cauchy Residue Theorem},
pdftitle = {Constructing measures with identical moments},
pdfpagemode = UseNone
}
\usepackage{graphicx,xspace,colortbl}
\usepackage{amsmath,amsthm,amsfonts,amssymb}
\usepackage{color}
\usepackage{enumerate}
\usepackage{fancybox}
\usepackage{epsfig}
\usepackage{subfig}
\usepackage{pdfsync}
    \def\qed{\hfill$\sqcap\kern-8.0pt\hbox{$\sqcup$}$\\}
\DeclareMathOperator{\re}{Re}
\DeclareMathOperator{\im}{Im}
    \def\r{{\mathbb R}}
    \def\c{{\mathbb C}} 
    \def\d{{\textnormal d}}
    \def\i{{\textnormal i}}

    \def\h{{\mathbb{H}}} 
    \newmuskip\pFqskip
\mathchardef\pFcomma=\mathcode`,

\newtheorem{theorem}{Theorem}

\newtheorem{proposition}{Proposition}

\theoremstyle{definition}
\newtheorem{definition}{Definition}

\newtheorem{remark}{Remark}

\title{Constructing measures with identical moments}
\author{
{Alexey Kuznetsov
\footnote{Dept. of Mathematics and Statistics,  York University,
4700 Keele Street, Toronto, ON, M3J 1P3, Canada.  \newline
E-mail:  kuznetsov@mathstat.yorku.ca } 
 }}
 
 \date{\today}

\begin{document}
\maketitle

\begin{abstract} 
The Nevanlinna parametrization establishes a bijection between the class of all measures having a prescribed set of moments and the class of Pick functions. The fact that all measures constructed through the Nevanlinna parametrization have identical moments follows from the theory of orthogonal  polynomials and continued fractions. In this paper we explore the opposite direction: we construct a set of measures and we show that they all have identical moments, and then we establish a Nevanlinna-type parametrization for this set of measures. Our construction does not require the theory of orthogonal polynomials and it exposes the analytic structure behind the Nevanlinna parametrization. 
\end{abstract}

{\vskip 0.15cm}

\section{Introduction}

Let $\nu(\d x)$ be a positive measure on $\r$, having finite moments 
\begin{equation}\label{moment_problem}
\int_{\r} x^{n} \nu(\d x)=v_n, \;\;\; n=0,1,2,\dots
\end{equation}
The Hamburger's moment problem is the following: given a sequence of moments $\{v_n\}_{n\ge 0}$, find the set of all measures 
$\nu$ on $\r$ satisfying  \eqref{moment_problem}. In the case when there exists a unique such measure, the moment problem is called 
{\it determinate}, otherwise it is called {\it indeterminate} and in this case the set of all solutions to the moment problem
can be described through the Nevanlinna parametrization. 

To present the Nevanlinna parametrization we will follow the review of Berg 
\cite{Berg1995}. The proofs of all these results can be found in the classical books by Akhiezer \cite{Akhiezer}
and Shohat and Tamarkin \cite{Shohat}. 

Denote by $\{P_n\}_{n\ge 0}$ the sequence of orthonormal polynomials with respect to the measure $\nu(\d x)$. We assume that $P_n$ is of degree $n$ with positive leading coefficient. Define the 
sequence of polynomials of the second kind as
$$
Q_n(x):=\int_{\r} \frac{P_n(x)-P_n(y)}{x-y} \nu(\d y), \;\;\; n=0,1,2,\dots 
$$
Next we define the functions $A,B,C$, $D$ as follows
\begin{align*}
\nonumber
A(z)&:=z \sum\limits_{n\ge 0} Q_n(0) Q_n(z),\\
B(z)&:=-1+z \sum\limits_{n\ge 0} Q_n(0) P_n(z),\\
\nonumber
C(z)&:=1+z \sum\limits_{n\ge 0} P_n(0) Q_n(z),\\
\nonumber
D(z)&:=z \sum\limits_{n\ge 0} P_n(0) P_n(z).
\end{align*}
It is known that in the indeterminate case the above series converge uniformly on compact subsets of $\c$ and the
functions $A$, $B$, $C$, $D$ are entire. These functions satisfy the following fundamental identity
\begin{equation}\label{eqn_ABCD}
A(z)D(z)-B(z)C(z)\equiv 1, \;\;\; z\in \c. 
\end{equation}

 We recall that {\it Pick functions} are defined as functions that are analytic in the upper halfplane
\begin{equation}\label{def_h_plus}
\h^+:=\{ z\in \c \; : \; \im(z)>0\}
\end{equation}
and that satisfy $f(\h^+) \subseteq \h^+$. Let us denote by ${\mathcal P}$ the family of all Pick functions.  
The Nevanlinna parametrization gives a bijection between ${\mathcal P} \cup \{\infty\}$ and
 the class of all solutions to the indeterminate moment problem: 
 $\nu(\d x)$ is a solution to the moment problem \eqref{moment_problem} if and only if  
\begin{equation}\label{Nevanlinna_parametrization}
\int_{\r} \frac{\nu(\d x)}{x-z} = -\frac{A(z) \phi(z) - C(z)}{B(z) \phi(z) - D(z)}, \;\;\;
{\textnormal{ for all }} z \in \h^+,
\end{equation}
where $\phi \in {\mathcal P} \cup \{\infty\}$. We will denote by $\nu(\d x; \phi)$ the measure corresponding to 
a Pick function $\phi$ in \eqref{Nevanlinna_parametrization}. This measure can be recovered from 
\eqref{Nevanlinna_parametrization} via {\it Stieltjes inversion}, which states that for any continuous $f: \r \mapsto \r$ with compact support we have
\begin{equation}\label{Stieltjes_inversion}
\int_{\r} f(x) \nu(\d x;\phi)= \lim\limits_{\epsilon \to 0^+} \frac{1}{\pi} 
\im \int_{\r} I(x+\i \epsilon) f(x) \d x, 
\end{equation}
where we have denoted
$$
I(z):=\int_{\r} \frac{\nu(\d x)}{x-z}, \;\;\; z \in \h^+. 
$$

In many cases the measure $\nu(\d x;\phi)$ in \eqref{Stieltjes_inversion} can be identified explicitly in terms of 
$\phi$ and the functions $B$ and $D$. For example,  assume that for $x\in \r$ we have 
$\phi(x+\i \epsilon)\to \phi(x)$ as $\epsilon \to 0^+$, uniformly in $x$ on compact subsets of $\r$, and also assume that 
$B(x) \phi(x)-D(x)\neq 0$ for $x \in \r$. Then \eqref{eqn_ABCD} and \eqref{Stieltjes_inversion} 
 imply that the measure $\nu(\d x;\phi)$ has a density given by 
\begin{equation}\label{nu_density}
\nu(x;\phi)=\frac{\im(\phi(x))/\pi}{|D(x)-\phi(x) B(x)|^2}, \;\; x\in \r.  
\end{equation}

In general the functions $A$, $B$, $C$, $D$ are not known in closed form, and it seems that there are no known sufficient conditions that ensure that four entire functions $a$, $b$, $c$, $d$ are the Nevanlinna functions 
$A$, $B$, $C$, $D$ for some indeterminate moment problem. There do exist several necessary conditions though: first of all, 
these functions must satisfy  the identity \eqref{eqn_ABCD}. Also, $-A(z)/B(z)$, $-C(z)/D(z)$ and $-D(z)/B(z)$ must be Pick functions (see \cite{Berg1995}). It is also known \cite{Berg1994} that the four functions $A$, $B$, $C$, $D$ are of minimal 
exponential type and have the same order and type (see \cite{Levin} for the definition and properties of the order and type of entire functions).  Several other necessary conditions can be found in \cite{Berg1995}. 

There exist a small number of indeterminate moment problems where the functions $A$, $B$, $C$, $D$, appearing in the Nevanlinna parametrization have been computed explicitly, see
\cite{Chen1998,Ch_2005,ChIs2006,Ismail_2001,IV1998}. Each such example leads to a family of explicit measures 
$\nu(x;\phi)$ that have identical moments. Typically this is done by taking the Pick function 
$\phi$ to be a constant. For example, if $\phi \equiv t \in \r \cup \{\infty\}$ then we obtain a family of discrete measures, which are 
called {\it N-extremal} 
(see \cite{Akhiezer, Berg1995, Shohat} for their properties) and if  
$\phi \equiv t+\i \gamma$, with $t\in \r$ and $\gamma >0$, then we obtain a family of measures having a smooth density 
as in \eqref{nu_density}. The following
interesting question was raised by Mourad Ismail: how one can show directly that all of the measures in
\eqref{nu_density} with $\phi\equiv t+ \i \gamma$ have identical moments? 
Of course one could show this by backtracking the derivation of the Nevanlinna parametrization, but this would not offer any new insight into the problem. One would hope that a direct analytical proof would expose the structure of the Nevanlinna parametrization and tell us something new about it. And this is precisely the main goal of the current paper: under quite general conditions we will show 
that the measures $\nu(x;\phi)$ defined via \eqref{nu_density} have identical moments. Our proof 
is entirely elementary and it uses only the Cauchy Residue Theorem.

This paper is organized as follows: in Section \ref{section_results} we present and prove our main results. In Section
\ref{section_example} we consider an explicit example coming from the paper of Ismail and Valent \cite{IV1998} and we discuss the connections between our construction and the classical Nevanlinna parametrization.

\section{Results}\label{section_results}

Let us first present the notation. The open upper half-plane $\h^+$ is defined via 
\eqref{def_h_plus}, and we denote the open lower half-plane by $\h^-:=\{z\in \c \; : \; \im(z)<0\}$. We will denote the open unit disk by ${\mathbb D}:=\{z\in \c \; : \; |z|<1\}$. Given a function $f : \Omega \to \c$ we will denote 
$\bar f(z):={\overline{f(\bar z)}}$. Note that the function $\bar f$ is defined on 
the domain $\bar \Omega:=\{z\in \c \; : \; \bar z \in \Omega\}$ and $\bar f$ is analytic on 
$\bar \Omega$ if the original function $f$ is analytic 
on $\Omega$.

The following class of entire functions will play an important role. 

\begin{definition}\label{def_M_class}
Let ${\mathcal M}$ be the class of entire functions $f$ such that
\begin{itemize}
 \item[(i)] $f(z)=C\prod_{n\ge 1} (1-z/z_n)$ where $z_n \in \h^+$, $n\ge 1$;
 \item[(iii)] for every $n \ge 0$ it is true that $z^n/f(z) \to 0$, as  $z\to \infty$,  uniformly in $\h^- \cup \r$.  
\end{itemize}
\end{definition}

Now we describe some simple properties of functions in the class ${\mathcal M}$. 

\begin{proposition}\label{prop1}
Assume that $f \in {\mathcal M}$. Let $f(z)=b(z)-\i d(z)$, where $b$ and $d$ are real entire functions. Then 
\begin{itemize}
\item[(i)] $f(z)/\bar f(z) \in {\mathbb D}$ for $z\in \h^+$;
\item[(ii)] $-d(z)/b(z)$ is a Pick function. 
\end{itemize}
\end{proposition}
\begin{proof}
This result can be found in Chapter 27 in \cite{Levin}, however the proof is very short and simple and we reproduce it here for convenience of the reader. The first statement follows from property (i) in Definition \ref{def_M_class}: we only need to check that 
$$
\Big | \frac{1-z/z_n}{1-z/\bar z_n} \Big|=\frac{|z-z_n|}{|z-\bar z_n|}<1
$$
for $z,z_n\in \h^+$, and this latter fact is obvious from geometric considerations. To prove the second statement, we write $\bar f(z)=b(z)+\i d(z)$ and we check that
$$
-\frac{d(z)}{b(z)}=\i \frac{1-f(z)/\bar f(z)}{1+f(z)/\bar f(z)}.
$$
Assume $z\in \h^+$. Using the following two facts: (i) the linear fractional 
transformation $\i (1-w)/(1+w)$ maps ${\mathbb D}$ onto $\h^+$, and  (ii) $f(z)/\bar f(z) \in {\mathbb D}$,
we conclude that  $-d(z)/b(z) \in \h^+$, thus $-d(z)/b(z)$ is a Pick function. 
\end{proof}

\begin{remark}
All functions  $f\in {\mathcal M}$ also belong to the class $P$ of entire functions, as defined on page 217 in Levin's book \cite{Levin}. This can be seen from item (i) of Proposition \ref{prop1} and Corollary 3 on page 218 in \cite{Levin}.
\end{remark}

\begin{definition}\label{def_Pick_subclass}
Let ${\mathcal P}_b$ denote the class of Pick functions $\phi$ that satisfy
\begin{itemize}
\item[(i)] $\lim_{\epsilon \to 0^+} \phi(x+\i \epsilon)=:\phi(x)$, uniformly in $x$ on compact subsets of $\r$;
\item[(ii)] the closure of $\phi(\h^+)$ is a bounded subset of $\h^+$.  
\end{itemize} 
\end{definition}

 As we establish in the next result, the class ${\mathcal P}_b$ is quite large. 

\begin{proposition}
The class ${\mathcal P}_b$ is dense in the class of all Pick functions ${\mathcal P}$. 
\end{proposition}
\begin{proof}
We need to prove that for any $\phi\in {\mathcal P}$, any $\epsilon>0$ and any compact $\Omega \subset \h^+$ there exists $\tilde \phi \in {\mathcal P}_b$ such that $|\phi(z)-\tilde \phi(z)|<\epsilon$ for all $z\in \Omega$. 

First, we check that 
for any $\delta>0$ the linear fractional function 
$$
g_{\delta}(z)=\frac{-1}{\i \delta - 1/(z+\i \delta)}
$$
maps $\h^+ \cup \r$ onto a disk inside $\h^+$, thus $g_{\delta}(z) \in {\mathcal P}_b$. 
It is easy to see that $g_{\delta}(z) \to z$ as $\delta \to 0^+$, uniformly on compact subsets of $\h^+$.
Given $\phi \in {\mathcal P}$ and any $\delta>0$, the function $z\mapsto \phi(z+\i \delta)$ is also in ${\mathcal P}$ and is satisfies condition (i). 
Thus, all functions $z\mapsto g_{\delta}(\phi(z+\i \delta))$ belong to ${\mathcal P}_b$ (for all $\delta>0$),
and these functions converge to $\phi(z)$ as 
$\delta \to 0^+$, uniformly in $z$ on compact subsets of $\h^+$. 
\end{proof}

Next, for $f\in {\mathcal M}$ and $\phi \in {\mathcal P}_b$ we define (as in \eqref{nu_density})
\begin{equation}\label{def_mu}
\mu(x;\phi,f):=\frac{\im(\phi(x))/\pi}{|d(x)-\phi(x) b(x)|^2}, \;\;\; x\in \r,
\end{equation}
where $f(x)=b(x)-\i d(x)$ and $b$, $d$, are real entire functions. Note that $\im(\phi(x))>0$ for $x\in \r$ (this follows from
condition (ii) in Definition \ref{def_Pick_subclass}), thus the denominator in \eqref{def_mu} is non-zero and the functions $\mu(x;\phi,f)$ are well-defined. In most cases we will keep $f \in {\mathcal M}$ fixed 
and we will suppress the dependence of $\mu$ on $f$: we will simply write $\mu(x;\phi)$. 

The next theorem is our first main result.
\begin{theorem}\label{theorem1}
Let us fix $f\in {\mathcal M}$. Then for every $\phi \in {\mathcal P}_b$ the moments 
$$
\mu_n=\int_{\r} x^n \mu(x;\phi) \d x, \;\;\; n = 0,1,2,\dots
$$
are finite and they do not depend on $\phi$. 
\end{theorem}

\vspace{0.25cm}
\noindent
{\bf Proof of Theorem \ref{theorem1}:}
Let us denote 
\begin{equation}\label{phi_to_w}
w(z)=\frac{1+\i \phi(z)}{1-\i \phi(z)}, \;\;\; z\in \h^+ \cup \r. 
\end{equation}
The linear fractional transformation $(1+\i z)/(1-\i z)$ maps the upper half-plane $\h^+$ onto the unit disk ${\mathbb D}$. Moreover, condition (ii) of Definition \ref{def_Pick_subclass} implies that
$w(\h^+ \cup \r)$ is a subset of ${\mathbb D}$, thus there exists $\delta \in (0,1)$ such that $|w(z)|<\delta$ for all $z \in \h^+ \cup \r$. 

One can check that \eqref{phi_to_w} implies 
\begin{equation}
\phi(z)=\i \frac{1-w(z)}{1+w(z)}, \;\;\; z \in \h^+ \cup \r. 
\end{equation} 
We plug this expression into \eqref{def_mu} and after some simplifications we rewrite \eqref{def_mu} in the equivalent form   
\begin{equation}\label{eqn_mu_in_terms_of_w}
\mu(x;\phi)=\frac{(1-|w(x)|^2)/\pi}{|w(x)f(x)-\bar f(x)|^2}, \;\;\; x\in \r. 
\end{equation}
As we will see later, this expression for $\mu(x;\phi)$ is more convenient to work with, compared to \eqref{def_mu}. 

Let us check first that all moments  of the measure $\mu(x;\phi) \d x$  are  finite. When $\phi(z) \equiv \i$ we get $w(z) \equiv 0$ from 
\eqref{phi_to_w}, thus   
$$
\mu(x;\i)=\frac{1}{\pi} |f(x)|^{-2}, \;\;\; x\in \r. 
$$
Condition (i) of Definition \ref{def_M_class} implies that $ f$ is zero-free on $\r$, and condition (ii) guarantees that
$f(x)$ grows faster than any power  of $x$ (as $x\to \infty$), thus  all moments of $\mu(x,\i)\d x$ are finite.
 In the general case $\phi \in {\mathcal P}_b$,  we have the two estimates
 $|w(z)|<\delta<1$ and $|f(z)/\bar f(z)|\le 1$ for any $z\in \h^+ \cup \r$, which were proved above, and these estimates give us   
\begin{align}\label{mu_estimate}
\mu(z;\phi)&=|\bar f(z)|^{-2} \frac{(1-|w(z)|^2)/\pi}{|1-w(z) f(z)/\bar f(z)|^2}\le |\bar f(z)|^{-2} \frac{1/\pi}{(1-\delta)^2}, 
\;\;\; z\in \h^+ \cup \r.
\end{align}
Thus, for $x\in \r$ we have $\mu(x;\phi)<(1-\delta)^{-2} \mu(x;\i)$, which shows that all measures $\mu(x,\phi)\d x$ have finite moments. 

Now we come to the main step in the proof of Theorem \ref{theorem1}. The main tool is the following identity
\begin{equation}\label{main_identity}
\frac{1-|w(x)|^2}{|w(x)f(x)-\bar f(x)|^2}=\frac{1}{f(x)\bar f(x)}
+\frac{w(x)}{\bar f(x)(\bar f(x)-w(x)f(x))}+\frac{\bar w(x)}{f(x)(f(x)-\bar w(x) \bar f(x))},
\;\;\; x\in \r,
\end{equation}
which we leave to the reader to verify. Using this identity and formula \eqref{eqn_mu_in_terms_of_w} we obtain
\begin{equation}\label{proof1}
\int_{\r} x^{n} \mu(x;\phi) \d x= \int_{\r} x^{n} \mu(x;\i) \d x
+\frac{2}{\pi} \re I_n,
\end{equation}
where we have denoted 
\begin{equation}\label{def_I_n}
I_n:=\int_{\r} \frac{x^n w(x) \d x}{\bar f(x)(\bar f(x)-w(x) f(x))}.
\end{equation}
Our goal is to prove that $I_n=0$. Informally, this is true because the integrand is analytic in $\h^+$ and when we move the contour of integration to $+\i \infty$ the integrand converges to zero very fast. Let us now present all the details. For $r>0$ we define the contour of integration 
\begin{equation}\label{def_L_R}
L_r:=(-\infty, -r] \cup S_r \cup [r,\infty),
\end{equation}
where 
$S_r:=\{r e^{\i \theta} \; : \; 0\le \pi \le \theta\}$ is a semicircle of radius $r$.  
This is a deformed real line, where the origin is traversed from above via $S_r$. The 
direction is from $-\infty$ to $+\infty$, so that the semicircle  $S_r$ is traversed clockwise.

Note that the integrand  in \eqref{def_I_n} is analytic in $\h^+$ and continuous on $\h^+ \cup \r$. This follows from the estimate $|1-w(z) f(z)/\bar f(z)|>1-\delta$ for $z\in \h^+ \cup \r$ and the fact that 
$\bar f$ does not have zeros in $\h^+ \cup \r$. Thus we can deform the contour of integration in \eqref{def_I_n} and obtain
\begin{align}
\label{eqn_I_n1}
I_n&=\int_{-\infty}^{-r} \frac{x^n w(x) \d x}{\bar f(x)(\bar f(x)-w(x) f(x))}
+\int_{S_r} \frac{z^n w(z) \d z}{\bar f(z)(\bar f(z)-w(z) f(z))}\\
\nonumber
&+\int_{r}^{\infty} \frac{x^n w(x) \d x}{\bar f(x)(\bar f(x)-w(x) f(x))}=:I^{(1)}_n(r)+I^{(2)}_n(r)+I^{(3)}_n(r). 
\end{align}
It is clear that $I^{(1)}_n(r)$ and $I^{(3)}_n(r)$ converge to zero as $r \to +\infty$ (since the integrand is integrable over $\r$). Using the estimate $|1-w(z) f(z)/\bar f(z)|>1-\delta$ for $z\in \h^+$, we can bound $I^{(2)}_n(r)$ as follows
\begin{align*}
|I^{(2)}_n(r)|\le
\int_{S_r} \frac{r^n |\d z|}{|\bar f(z)|^2 |1-w(z) f(z)/\bar f(z)|}<
 \frac{r^{n+1}}{1-\delta}  \int_0^{\pi} \frac{ \d \theta}
{|\bar f(re^{\i \theta})|^2}, 
\end{align*}
and the quantity in the right-hand side converges to zero as $r\to +\infty$, due to condition (ii) of Definition \ref{def_M_class}.
Thus we have proved that $I_n=0$, and formula \eqref{proof1} implies that all measures $\mu(x;\phi)\d x$ (with $\phi \in {\mathcal P}_b$) have identical moments. 
\qed

 Let us now fix $f \in {\mathcal M}$ and let us denote 
\begin{equation}
\label{def_g}
g(z):=- f(z)  \int_{\r} \frac{\mu(x;\i) \d x}{x-z}
=-\frac{f(z)}{\pi}  \int_{\r} \frac{\d x}{f(x) \bar f(x)(x-z)},  \;\;\; z\in \h^-. 
\end{equation}
We recall that we denoted by $b$ and $d$ the real entire functions such that $f(z)=b(z)-\i d(z)$.
The next theorem is our second main result: it establishes a Nevanlinna-type parametrization for the family of measures 
$\mu(x;\phi)\d x$, $\phi \in {\mathcal P}_b$.

\begin{theorem}\label{theorem2}
Let us fix $f\in {\mathcal M}$ and define $g$ as in \eqref{def_g}. 
\begin{itemize}
\item[(i)] The function $g$ can be analytically continued to an entire function.
\item[(ii)] Let $g(z)=a(z)-\i c(z)$ where $a$ and $c$ are real entire functions. Then the following identity is true
\begin{equation}\label{adbc_identity}
a(z)d(z)-b(z)c(z)=1, \;\;\; z\in \c.
\end{equation}
\item[(iii)] For all $\phi \in {\mathcal P}_b$ we have 
\begin{equation}\label{Nevanlinna_parametrization_2}
\int_{\r} \frac{\mu(x;\phi)\d x}{x-z} = -\frac{ a(z)\phi(z)-c(z)}{ b(z)\phi(z)-d(z)},\;\;\; z \in \h^+. 
\end{equation}
\end{itemize}
\end{theorem}
\begin{proof}
Let us first prove item (i). 
 Since $f$ has no zeros on $\r$, we can find $r>0$ small enough, such that $f(u)\bar f(u)\neq 0$ 
when $|u|<2r$ (this follows from the fact that zeros of non trivial entire functions can not have finite accumulation points). Therefore, we can fix $z \in \h^-$ and deform the contour of integration $\r \mapsto L_r$ in the integral in \eqref{def_g},
where $L_r$ was defined above in \eqref{def_L_R}. Thus we obtain 
\begin{equation}\label{g_deformed_contour}
g(z)= -\frac{f(z)}{\pi} \int_{L_r} \frac{\d u}{f(u)\bar f(u)(u-z)}.
\end{equation}
The above equation gives us an analytic continuation of $g(z)$ into the domain $\h^- \cup \{|z|<r\}$. 
Now we fix any $z$ satisfying $|z|<r$ and $z\in \h^+$ and we deform the contour of integration in 
\eqref{g_deformed_contour} in the opposite direction: $L_r \mapsto \r$. 
Note that now the integrand is not analytic  ``in between" the old contour $L_r$ and the new
contour $\r$: it has a simple pole at $u=z$. Applying the Cauchy Residue Theorem we obtain
\begin{align}\label{H_lower_half_plane}
\nonumber
g(z)&=\frac{f(z)}{\pi} \times 2\pi \i \times {\textnormal{Res}} \Big( 
\frac{1}{f(u)\bar f(u)(u-z)} \; \Big | \; u=z \Big)
-\frac{f(z)}{\pi} \int_{\r} \frac{\d u}{f(u)\bar f(u)(u-z)}\\
&=\frac{2 \i}{\bar f(z)}- f(z)\int_{\r} \frac{\mu(x;\i) \d x}{x-z}=
\frac{2 \i}{\bar f(z)}+ \frac{f(z)}{\bar f(z)} \bar g(z).
\end{align}
In the last step of the above computation we have used the fact that 
\begin{equation}\label{g_hat}
\bar g(z)= -\bar f(z) \int_{\r} \frac{\mu(x;\i) \d x}{x-z}, \;\;\; z\in \h^+. 
\end{equation}
We remind the reader that formula \eqref{H_lower_half_plane}  holds for any $z\in \h^+$, such that $|z|<r$. 
However, the right-hand side in \eqref{H_lower_half_plane} is analytic everywhere in $\h^+$ (since 
$\bar g(z)$ is analytic in $\h^+$ and $\bar f(z)$ is zero-free in $\h^+$). Thus we have obtained an analytic continuation
of $g(z)$ to the domain $\h^- \cup (-r,r) \cup \h^+$. 

The same procedure can be repeated for any point $a \in \r$: instead of deforming the contour of integration so that we bypass 
zero by a semicircle from above, we can bypass $a$ by a small semicircle from above, and obtain an analytic continuation to the domain $\h^- \cup (a-r,a+r) \cup \h^+$ (for some $r$ depending on $a$). Thus 
$g(z)$ is analytic in the entire complex plane and we have proved item (i).

While doing the above calculations, we have also proved item (ii). Indeed, formula \eqref{H_lower_half_plane} implies the identity
\begin{equation}\label{fg_identity}
g(z) \bar f(z)-f(z) \bar g(z)=2\i,
\end{equation}
which is easily seen to be equivalent to \eqref{adbc_identity}. 

It remains to prove item (iii). Again, our main tool is the key identity  \eqref{main_identity}: from this result we obtain for $z\in \h^+$
\begin{align}\label{proof2}
\nonumber
 \int_{\r} \frac{\mu(x;\phi)\d x}{x-z} 
=\int_{\r} \frac{\mu(x;\i) \d x}{x-z}
&+\int_{\r} \frac{w(x)/\pi}{\bar f(x)(\bar f(x)-w(x)f(x))}
\times \frac{\d x}{x-z}\\
&+
 \int_{\r} \frac{\bar w(x)/\pi }{f(x)(f(x)-\bar w(x) \bar f(x))}
 \times \frac{\d x}{x-z}=:J_1+J_2+J_3.
\end{align}
Our first goal is to prove that $J_3=0$. Consider its conjugate
$$
\bar J_3=
\int_{\r} \frac{w(x)/\pi }{\bar f(x)(\bar f(x)- w(x)  f(x))}
 \times \frac{\d x}{x-\bar z}. 
$$
Note that $\bar z \in \h^-$, thus the function $x\mapsto 1/(x-\bar z)$ is analytic in $\h^+$. Now we can repeat verbatim the argument in the proof of Theorem \ref{theorem1}, where we have demonstrated that the integral $I_n$ in \eqref{def_I_n} is identically equal to zero, and we can deduce in exactly the same way that $\bar J_3=0$.

We can also use the same argument to deal with the integral $I_n$ in \eqref{def_I_n} and show that
\begin{equation}\label{eqn_J2}
J_2=\int_{\r} \frac{w(x)/\pi}{\bar f(x)(\bar f(x)-w(x)f(x))}
\times \frac{\d x}{x-z}=\frac{2\i w(z)}{\bar f(z)(\bar f(z)-w f(z))}. 
\end{equation}
Indeed, since $z\in \h^+$, the integrand in \eqref{eqn_J2} has a simple pole at $x=z$, so that when we shift the contour of integration $\r \mapsto L_r$ (as in  \eqref{eqn_I_n1}) and take $r\to +\infty$, we need to take into account the residue at $x=z$, and this residue gives us the right-hand side of \eqref{eqn_J2}.    

Combining \eqref{g_hat}, \eqref{proof2} and \eqref{eqn_J2} and the fact that $J_3=0$ we obtain
\begin{align}\label{proof3}
\nonumber
\int_{\r} \frac{\mu(x;\phi)\d x}{x-z} 
&=\int_{\r} \frac{\mu(x;\i) \d x}{x-z}
+\frac{2\i w(z)}{\bar f(z)(\bar f(z)-w(z) f(z))}\\
&=-\frac{\bar g(z)}{\bar f(z)}+\frac{2\i w(z)}{\bar f(z)(\bar f(z)-w(z) f(z))}
\\
\nonumber
&=-\frac{1}{w(z) f(z)-\bar f(z)} \Big[ 
w(z) \Big(\frac{2\i}{\bar f(z)}+\frac{f(z)}{\bar f(z)} \bar g(z)  \Big)-\bar g(z) 
\Big]
=-\frac{w(z) g(z)-\bar g(z)}{w(z) f(z)-\bar f(z)},
\end{align}
where in the last step we have also used \eqref{H_lower_half_plane}. 
The above formula is equivalent 
to \eqref{Nevanlinna_parametrization_2} after we express $w(z)$ in terms of $\phi(z)$ via 
\eqref{phi_to_w} and substitute $f(z)=b(z)-\i d(z)$ and $g(z)=a(z)-\i c(z)$. 
\end{proof}

\begin{remark}
Given the fact that $g(z)$ is entire, the identity 
\eqref{fg_identity} (and thus \eqref{adbc_identity}) 
can be obtained as a simple consequence of Stieljes inversion
\eqref{Stieltjes_inversion}. Indeed, when $z\in \r$ we use \eqref{g_hat} and obtain 
\begin{align*}
\frac{1}{2\i} \Big[\frac{g(z)}{f(z)}-\frac{\bar g(z)}{\bar f(z)} \Big]&=
-\im \frac{\bar g(z)}{\bar f(z)}=
-\lim_{\epsilon \to 0^+}  \im \frac{\bar g(z+\i \epsilon)}{\bar f(z+\i \epsilon)}\\
&=
\lim_{\epsilon \to 0^+}  \frac{1}{\pi} \im \int_{\r} \frac{1}{f(x)\bar f(x)}   \times \frac{\d x}{x-z-\i \epsilon}
=\frac{1}{f(z)\bar f(z)}.
\end{align*} 
This result can be extended to $z\in \c$ by analytic continuation. 
\end{remark}

\begin{remark}
Formula \eqref{proof3} can be used to give an alternative proof of the fact that 
the measures $\mu(x,\phi)\d x$ (for all $\phi \in {\mathcal P}_b$) have identical moments. Indeed, considering the asymptotic expansion of both sides of \eqref{proof3} as $\im(z) \to +\infty$ we see that 
\begin{equation}\label{asymptotic_eqn}
\sum\limits_{n\ge 0} z^{-n-1} \int_{\r} x^n \mu(x;\phi)\d x 
=\sum\limits_{n\ge 0} z^{-n-1} \int_{\r} x^n \mu(x;\i) \d x
+F(z), \;\;\; z\to +\i \infty, 
\end{equation}
where 
$$
F(z):=-\frac{2\i w(z)}{\bar f(z)^2(1-w(z) f(z)/\bar f(z))}. 
$$
Since $|w(z)|<\delta$ and $|1-w(z) f(z)/\bar f(z)|>|1-\delta|$ for $z\in \h^+$, condition (ii) of Definition \ref{def_M_class}
implies that for every fixed $N\ge 0$ it is true that $F(z)=o(z^{-N})$ as $\im(z) \to +\infty$. 
Thus $F(z)$ is asymptotically smaller than any power of $z$, and the coefficients in front of $z^{-n-1}$ in the asymptotic expansion \eqref{asymptotic_eqn} must coincide, which proves that the  moments of $\mu(x;\phi)\d x$ coincide with those 
of $\mu(x;\i)\d x$.  
\end{remark}

\section{Connections with the Nevanlinna parameterization}\label{section_example}

It is natural to ask
what is the connection between the functions $A$, $B$, $C$, $D$ (which appear in the Nevanlinna parametrization 
\eqref{Nevanlinna_parametrization}) and the functions $a$, $b$, $c$, $d$ (which appear in the 
Nevanlinna-type parametrization \eqref{Nevanlinna_parametrization_2}). 
It turns out that this is a not trivial question. As we will demonstrate in the next example, two situations can arise:

\label{page1}

\vspace{0.25cm}
\noindent
{\bf Case I:} The functions $B$, $D$ are linear combinations of $b$ and $d$.

\vspace{0.15cm}
\noindent
{\bf Case II:}
The functions $B$, $D$ can not be obtained as linear combinations of $b$ and $d$.
 
\vspace{0.25cm}

To demonstrate this, we will consider an indeterminate moment problem studied by Ismail and Valent in \cite{IV1998}. 
Let us fix $k\in (0,1)$ and denote the complete elliptic integral of the first kind by 
$$
K(k):=\frac{\pi}{2} \times  {}_2F_1(\tfrac{1}{2},\tfrac{1}{2};1;k^2).
$$   
We also denote $K:=K(k)$, $k':=\sqrt{1-k^2}$ and $K':=K(k')$.
Take $f(z)=\tfrac{2}{\sqrt{\pi}}\cos(\sqrt{z}(K-\i K')/2)$, so that $f(z)=b(z)-\i d(z)$, where 
\begin{equation}\label{eqn_b_d}
b(z)=\frac{2}{\sqrt{\pi}}\cos(\sqrt{z}K/2)\cosh(\sqrt{z}K'/2), \;\;\; d(z)=-\frac{2}{\sqrt{\pi}}\sin(\sqrt{z}K/2)\sinh(\sqrt{z}K'/2). 
\end{equation}
It is an easy exercise to show that $f\in {\mathcal M}$. 
Consider the function
\begin{equation}\label{mu_IV}
\mu(x;\i)=\frac{1/\pi}{|f(x)|^2}=\frac{1/2}{\cos(\sqrt{z}K)+\cosh(\sqrt{z}K')}. 
\end{equation}
As was shown by Ismail and Valent in \cite{IV1998}, the measure $\mu(x;\i) \d x$ has total mass one 
(see also \cite{Kuznetsov} for an analytical proof of this result) and this measure is a solution to a certain indeterminate moment problem. 
The functions $B(x)$ and $D(x)$, appearing in the Nevanlinna parametrization of this indeterminate moment problem, 
were computed explicitly in \cite{IV1998}[Theorem 4.4], and are given by  
\begin{equation}\label{eqn_B_D}
B(x)=\frac{\sqrt{\pi}}2 b(x)-\frac{1}{\sqrt{\pi}} \ln(k/k') d(x), \;\;\;  
D(x)=\frac{2}{\sqrt{\pi}} d(x).
\end{equation}
We see that $B$, $D$ are linear combinations of $b$, $d$ and we are in Case I. 

Let us now consider the function
$$
\tilde f(z):=f(z)+\i z b(z). 
$$
One can show that $\tilde f \in {\mathcal M}$, thus we obtain a family of measures $\mu(x;\phi,\tilde f)\d x$, $\phi \in {\mathcal P}_b$ with identical moments. Now we ask the same question: What is the Nevanlinna parametrization for this indeterminate moment problem? The answer is that the functions $B$ and $D$ in the Nevanlinna parametrization are the same as in \eqref{eqn_B_D}.
In particular, they can not be expresses as the linear combination of $\tilde b(z)=b(z)$ and 
$\tilde d(z)=d(z)-zb(z)$ (where $\tilde f(z)=\tilde b(z)-\i \tilde d(z)$), and we are in Case II. 

Let us explain how we arrived at this conclusion. Consider a measure $\nu(x;\tilde \phi)\d x$ obtained via \eqref{nu_density}
with the Pick function 
\begin{equation}\label{tilde_phi}
\tilde \phi(z):=\frac{4}{\pi}(z+\i)
\end{equation}
 and $B$ and $D$ as in \eqref{Nevanlinna_parametrization}.  
A simple calculation shows that 
$$
\nu(x;\tilde \phi)=\frac{1/\pi}{|\tilde f(x)|^2}=\mu(x;\i,\tilde f), \;\;\; x\in \r.  
$$  
However, since this measure was constructed through the Nevanlinna parametrization \eqref{Nevanlinna_parametrization}, it must have the same moments as measure \eqref{mu_IV}, thus the functions $B$ and $D$ in the Nevanlinna parametrization must coincide 
with those given in \eqref{eqn_B_D}. 

Thus we have demonstrated that some indeterminate moment problems constructed through \eqref{def_mu} fall into case (I) and
some fall into case (II) (as described on page \pageref{page1}).  
It is an interesting problem to find what properties of a function $f \in {\mathcal M}$ 
would allow one to distinguish between these two cases.

\section*{Acknowledgements}

 The research was supported by the Natural Sciences and Engineering Research Council of Canada. 
  We would like to thank Mourad Ismail for helpful discussions.

%

\end{document}